\documentclass[12pt]{amsart}
\usepackage{graphicx}
\usepackage{amsmath, amsthm, amssymb,amsfonts, mathrsfs}
\usepackage[cal=euler]{mathalpha}
\usepackage{mathtools}
\usepackage{tikz}
\usetikzlibrary{decorations.pathreplacing,calligraphy,calc}
\usepackage{tikz-cd}
\usepackage{color} 
\definecolor{rossred}{rgb}{1.0,0.25,0.66}  
\definecolor{rossgreen}{rgb}{0.25,0.66,0.25} 
\definecolor{rossblue}{rgb}{0.25,0.66,1.0}
\definecolor{sashapurple}{rgb}{0.5,0.15,0.5}
\usepackage[pdfborder={0 0 0}]{hyperref}
\hypersetup{
    colorlinks=true,
    linkcolor=rossblue,
    citecolor=rossgreen
}
\usepackage{url}
\usepackage{enumitem}
\usepackage[paperwidth = 8.5in, paperheight = 11in, inner = 1in, outer = 1in, top = 1in, bottom = 1in]{geometry}
\usepackage{graphicx}%
\usepackage{multirow}%
\usepackage{amsmath,amssymb,amsfonts}%
\usepackage{amsthm}%
\usepackage{mathrsfs}%
\usepackage[title]{appendix}%
\usepackage{xcolor}%
\usepackage{textcomp}%
\usepackage{manyfoot}%
\usepackage{booktabs}%
\usepackage{algorithmicx}%
\usepackage{algpseudocode}%
\usepackage{listings}%
\usepackage{tikz}
\usepackage{xspace}
\usepackage{tkz-graph}
\usepackage{caption}
\usepackage{cancel}
\usepackage{comment}
\usepackage{enumitem}
\usepackage[T1]{fontenc}
\usepackage{lmodern}
\usepackage{mathrsfs} 
\usepackage{caption}
\numberwithin{equation}{section}
\theoremstyle{plain}
\newtheorem{theorem}{Theorem}[section]

\newtheorem{lemma}[theorem]{Lemma}

\newtheorem{corollary}[theorem]{Corollary}

\theoremstyle{definition}
\newtheorem{definition}[theorem]{Definition}

\newtheorem{remark}[theorem]{Remark}

\newtheorem{example}[theorem]{Example}
\AtBeginEnvironment{example}{%
  \pushQED{\qed}%
}
\AtEndEnvironment{example}{\popQED\endexample}

\title{The Unit-Zero Divisor Graph of a Commutative Ring}
\author{Vika Yugi Kurniawan}
\address
{Department of Mathematics,
Universitas Gadjah Mada, Yogyakarta, Indonesia}
\address
{Department of Mathematics,
Universitas Sebelas Maret, Surakarta, Indonesia}
\email{vikayugikurniawan@mail.ugm.ac.id}

\author{Yeni Susanti}
\address
{Department of Mathematics,
Universitas Gadjah Mada, Yogyakarta, Indonesia}
\email{yeni\_math@ugm.ac.id}

\author{Budi Surodjo}
\address
{Department of Mathematics,
Universitas Gadjah Mada, Yogyakarta, Indonesia}
\email{surodjo\_b@ugm.ac.id}

\date{\today}

\begin{document}

\begin{abstract}
This paper introduces a new approach to associating a graph with a commutative ring. Let $R$ be a commutative ring with identity. The unit-zero divisor graph of a commutative ring $R$, denoted by $G_{UZ}(R)$, offers a novel framework for exploring the interaction between ring and graph structures. The vertex set of $G_{UZ}(R)$ consists of all elements of the ring $R$. Two distinct vertices $x$ and $y$ in $G_{UZ}(R)$ are adjacent if and only if $x + y$ is a unit and $xy$ is a zero divisor in $R$. This dual adjacency condition gives rise to a graph that reflects both the additive and multiplicative behavior of the ring. This study investigates key structural properties of $G_{UZ}(R)$, including regularity, bipartiteness, planarity, and Hamiltonicity. In addition, it examines how these graph features are influenced by the algebraic structure of the ring, particularly the group of units, the set of zero divisors, ideals, and the Jacobson radical.
\end{abstract}

\maketitle

\section{Introduction}\label{sec:intro}

The investigation of graph theory in connection with algebraic structures, particularly commutative rings, has undergone considerable advancement since its initial development. Algebraic graph theory provides a robust framework for understanding algebraic structures by associating elements of a ring with vertices and defining edges based on specific algebraic relations. A pivotal concept within this domain is the zero-divisor graph, which was initially proposed by Beck in \cite{Beck1988} as part of his study on the structure of commutative rings. While Beck's original formulation included all elements of the ring as vertices, the definition was later refined by Anderson and Livingston \cite{Anderson1999}, who focused solely on the nonzero zero-divisors as vertices. In \cite{Anderson1999}, two distinct vertices $x$ and $y$ are adjacent if and only if their product equals zero, i.e., $ xy = 0$. This definition not only elucidates the relationship between algebra and graph theory but also lays the groundwork for further scholarly inquiry into such graphs' characteristics and potential applications. Subsequent scholarship has built upon Beck's foundational framework, culminating in a more profound comprehension of the structural attributes inherent in zero-divisor graphs. For instance, investigations have concentrated on the Laplacian eigenvalues associated with these graphs, thereby uncovering substantial insights into their spectral properties \cite{chattopadhyay2020laplacian}. Furthermore, the concept of zero-divisor graphs has been generalized to encompass various categories of rings, such as semirings \cite{DOLAN2012} and commutative semirings \cite{EbrahimiAtani2009}, thus expanding the breadth of their applicability. The ramifications of these investigations permeate domains such as molecular structure \cite{Ahmadini2020}, combinatorial optimization \cite{Radha2021}, and computational algorithm \cite{Rather2022}.

Simultaneously, investigating unit graphs related to commutative rings has emerged as an essential area of scholarly inquiry. The notion of unit graphs pertinent to rings has attracted considerable focus within mathematical research, especially following Ashrafi's seminal contributions in 2010 \cite{Ashrafi2010}. A unit graph corresponding to a ring is formulated by designating vertices for all elements within the ring, where edges link two vertices if their summation constitutes a unit in the ring. This framework facilitates the examination of various algebraic attributes and interrelations inherent to the units of the ring. Ashrafi's research notably illuminated the bipartite property of unit graphs, showing that such graphs exhibit bipartiteness under specific algebraic conditions, an observation that plays a key role in understanding their structural behavior within graph theory. Building on this, Su and Wei \cite{Su2019} conducted a more focused investigation on the diameter of unit graphs, offering valuable insight into how the algebraic nature of the underlying ring affects graph-theoretic distance properties. This line of research deepens our understanding of the interplay between ring structures and graph metrics. Further geometric aspects were explored by Tang et al. \cite{Tang&Su2015}, who examined the planarity of unit graphs, shedding light on their spatial configurations and implications for graph visualization.  A recent advancement in this study, Susanti et al. \cite{Susanti2025} proposed the notion of a unit regular graph over finite rings as a generalization of the unit graph. In this graph, the vertices correspond to ring elements, and two vertices are adjacent if their sum is a unit regular element. Their investigation extends the framework of unit-based graphs by analyzing structural properties such as completeness, vertex degrees, girth, matching and independence numbers, and conditions for Eulerian circuit and Hamiltonian cycles. 

Beyond their theoretical relevance, unit graphs exhibit practical applications in various domains, including computer science, social networks, and nanotechnology. For instance, the construction of linear codes using incidence matrices of unit graphs, as demonstrated in \cite{Annamalai2021LinearCode}, highlights how algebraic graph structures derived from ring units can be utilized in coding theory applications. In another direction, Asir et al. \cite{Asir2022Wiener} investigated the Wiener and hyper-Wiener indices of unit graphs and explicitly determined their values, revealing potential applications in molecular chemistry where these indices are used to model molecular stability and branching. Moreover, Fakieh et al. \cite{Fakieh2024spectrum} investigated the spectral properties associated with unit graphs, which has facilitated new opportunities for their application in the analysis of network stability and dynamics.

Many different types of algebraic graphs are developed from these two basic concepts, each offering unique insights into ring theory. But until now, research on algebraic graphs has generally considered only one of the ring operations, addition or multiplication, as the basis for defining adjacency between vertices. For example, Anderson and Badawi \cite{AndersonBadawi2008total} define the total graph of a commutative ring $R$, denoted by $T(\Gamma(R))$. It is a graph with $R$ as the vertex set, and two distinct vertices $u$ and $v$ are adjacent if and only if their sum is a zero divisor. Nikhmer and Khojasteh \cite{NIKMEHR2013nilpoten} define the nilpotent graph of a ring. The nilpotent graph of the ring $R$, denoted as $\Gamma_N(R)$, is a graph whose vertex set $Z_N(R)=\{0\neq x \in R | (\exists y\in R\setminus\{0\}) \  xy\in N(R)\}$, where two distinct vertices $x$ and $y$ are adjacent if and only if $xy$ is a nilpotent element of $R$. Mohammad and Shuker \cite{Mohammad2022idempdivisor} define a new type of graph called the idempotent divisor graph of a ring $R$. The vertex set of this graph consists of non-zero elements in \emph{R}, and two vertices $v_1$ and $v_2$ are adjacent if and only if $v_1v_2=e$, where $e$ is a non-unit idempotent element. And there are still many types of graphs associated with rings that only consider one operation on the ring as the condition for vertex adjacency, such as the idempotent graph of a ring \cite{AKBARI-idempotent2013}, regular graph of a commutative ring \cite{Akbari2013regular}, non-nilpotent graph of commutative rings \cite{non-nilpoten2024}, tripotent divisor graph of a commutative ring \cite{Khaleel2024tripotent}, and non-zero graph of a commutative ring \cite{Nawaf2025}.

The representation of those graphs has been quite helpful in visualizing the relationships between elements in the ring. However, considering that the structure of a ring is fundamentally determined by two interrelated operations, namely addition and multiplication, there is a possibility that more in-depth structural information has not been fully represented in the generated graph. A graph representation incorporating both ring operations as adjacency criteria is expected to offer a deeper insight into the ring's structure and the interactions among its elements. Moreover, it opens the possibility to explore how the combination of these two operations influences the combinatorial properties of the resulting graph. This research can pave the way for new directions in algebraic graph theory, including the discovery of novel structural properties, spectral characteristics, and their applications in network theory and cryptography.

In this paper, we introduce a new graph associated with a commutative ring, called unit-zero divisor graph, by modifying the adjacency condition of the unit graph. Unlike earlier constructions which typically rely on a single operation of the ring, this graph incorporates both addition and multiplication to define adjacency. Specifically, for a commutative ring $R$ with identity, the unit-zero divisor graph of $R$, denoted by $G_{UZ}(R)$, is the graph whose vertices are all elements of $R$, with two distinct vertices $u$ and $v$ adjacent if and only if $u+v$ is a unit and $uv$ is a zero divisor. This dual-operation approach allows for a richer and more comprehensive representation of the ring structure. Beyond this algebraic motivation, another key reason for studying the unit-zero divisor graph lies in its potential applications in coding theory and cryptography. Previous studies have demonstrated that unit graphs of commutative rings may exhibit bipartite structures under certain algebraic conditions. By strengthening the adjacency criterion in $G_{UZ}(R)$, we obtain a graph that preserves bipartiteness while exhibiting greater sparsity. Graphs with such properties are widely regarded in the coding theory and cryptography areas. For example, Wang et al. \cite{Wang2019SparseBipartit} utilize sparse bipartite graphs as the underlying structure for the belief propagation (BP) decoder in 5G LDPC and polar codes. These structures enable maximization of hardware reuse and support efficient error correction. In the context of cryptography, Wang et al. \cite{Wang2024sparse} highlight that sparse graphs contribute to lower computational complexity and stronger resistance against certain types of attacks. Additionally, Chen et al. \cite{Chen2024bipartitcode} and Wu et al. \cite{Wu2025bipartitEncrypt} emphasize the importance of bipartite graphs in constructing lightweight and efficient cryptographic and coding systems.

This article presents a comprehensive study of the unit-zero divisor graph associated with commutative rings with identity, examining both general structural properties and specific behaviors in particular rings, such as local rings and rings of the form $\mathbb{Z}_{n}$. The investigation focuses on key graph-theoretic parameters and properties, including regularity, planarity, bipartiteness, Hamiltonicity, and Eulerianity, each of which is analyzed in relation to the underlying algebraic structure of the ring. Special emphasis is placed on how ideals, units, zero divisors, and the Jacobson radical influence the structure and characteristics of the graph. 

\section{Basic Notations and Preliminaries }

This section reviews several fundamental concepts and results from commutative ring theory and graph theory that will be used throughout the paper. Unless stated otherwise, all rings considered are finite, commutative, and possess a multiplicative identity. 
We begin with a basic but important observation: in any finite commutative ring $R$, every nonzero element is either a unit or a zero divisor. The set of units and the set of zero divisors of the ring $R$ are denoted by $U(R)$ and $Z(R)$, respectively. We now consider the relationship between ideals and the set of zero divisors. In a finite commutative ring, any proper ideal $I$ is necessarily contained in $Z(R)$. This follows from the fact that elements of a proper ideal cannot be invertible, and in the finite setting, non-units must be zero divisors.
It is also well known that in any commutative ring, every non-unit element is contained in some maximal ideals. This fact plays a crucial role in localizing ring behavior and is foundational to much of the structure theory in commutative algebra.

Recall that a commutative ring $R$ with identity is said to be \emph{local} if it contains a unique maximal ideal. Equivalently, $R$ is local if and only if the set of all non-units forms an ideal. This alternative characterization offers a useful criterion for identifying local rings and connects neatly with the graph-theoretic properties we will study. Finally, for finite local rings, there is a particularly elegant identity: the maximal ideal coincides with the set of all zero divisors. That is, if $R$ is a finite local commutative ring with maximal ideal $\mathfrak{m}$, then $\mathfrak{m} = Z(R)$.

In order to investigate the structural properties of the unit-zero divisor graph associated with a commutative ring, we utilize several classical results from graph theory. These results provide useful criteria for analyzing graph parameters such as Eulerian circuit and Hamiltonian cycles, as well as planarity.

One of the fundamental results we use is Euler's Theorem, which characterizes when a connected graph admits an Eulerian circuit. It states that a nontrivial connected graph $G$ is Eulerian if and only if every vertex of $G$ has an even degree. This theorem is particularly relevant when we analyze the degree structure of $G_{UZ}(R)$ in the context of finite rings.

We also consider the important properties related to Hamiltonian cycles. We refer to a well-known result that a complete bipartite graph $K_{m,n}$ is Hamiltonian if and only if $m = n \geq 2$. This condition ensures the existence of a Hamiltonian cycle that alternates between the two partitions without repetition, which is only possible when both partitions have the same number of vertices.

To analyze the planarity of $G_{UZ}(R)$, we appeal to a foundational result known as Kuratowski's Theorem. Before stating the theorem, we recall the concept of a graph subdivision. A graph $H$ is called a \textit{subdivision} of a graph $G$ if $H \cong G$, or $H$ can be obtained by inserting one or more vertices of degree two along some of the edges of $G$. Kuratowski's Theorem then characterizes planar graphs in terms of forbidden subdivisions. A graph is planar if and only if it does not contain a subdivision of the complete graph $K_5$ or the complete bipartite graph $K_{3,3}$ as a subgraph.

Next, we present a classical theorem that characterizes when a bipartite graph contains a complete bipartite subgraph. The theorem is the  \textit{Kóvári-Sós-Turán Theorem (KST Theorem)}, which provides an upper bound on the number of edges in bipartite graphs that exclude a complete bipartite subgraph $K_{s,t}$ as a subgraph. Specifically, let $G = (X \cup Y, E(G))$ be a simple bipartite graph, where the partite sets $X$ and $Y$ contain $m$ and $n$ vertices, respectively. For any integers $s \geq 1$ and $t \geq 1$, if $G$ contains no complete bipartite subgraph $K_{s,t}$, then the number of edges in $G$ satisfies the inequality as follows
\[
|E(G)| \leq (t - 1)^{1/s} \cdot m \cdot n^{1 - 1/s} + (s - 1)n.
\]
This theorem plays a key role in identifying structural properties of bipartite graphs and is often used to derive sufficient conditions for the existence of small cycles, such as 4-cycles, within bipartite graphs.

We now formally define the central object of this study, namely the \emph{unit-zero divisor graph} of a commutative ring. This graph-theoretic construction encodes algebraic interactions between units and zero divisors within the ring and will serve as the main framework for the results discussed in this paper.

\begin{definition}
Let $R$ be a finite commutative ring with identity, $U(R)$ be the set of units in $R$, and $Z(R)$ be the set of zero divisors of $R$. The unit-zero divisor graph of $R$, denoted $G_{UZ}(R)$, is the graph whose vertex set is the set of all elements of $R$, where two distinct vertices $u, v \in R$ are adjacent if and only if $u+v \in U(R)$ and $uv \in Z(R)$.
\end{definition}

By the definition, we obtain that the unit-zero divisor graph of a commutative ring is always a simple undirected graph. The commutativity of both addition and multiplication ensures that the adjacency is symmetric. Furthermore, the binary operations in a ring are well-defined, hence the conditions for adjacency are uniquely determined for each pair of distinct elements. Next, assume for contradiction that a vertex $u \in R$ is adjacent to itself. Then it must satisfy both $2u \in U(R)$ and $u^2 \in Z(R)$. But if $2u \in U(R)$, then $(2u)^2 = 4u^2 \in U(R)$. On the other hand, since $u^2 \in Z(R)$, there exists a nonzero $y \in R$ such that $u^2y = 0$. This leads to $4(u^2y) = 4u^2y = 0$, implying $4u^2 \in Z(R)$, contradicting the fact that $4u^2 \in U(R)$. Therefore, the unit-zero divisor graph contains neither loops nor multiple edges.

\section{Results}
In this section, we present several fundamental results concerning the structure and properties of the unit-zero divisor graph $G_{UZ}(R)$ of a commutative ring $R$. The first investigation is to determine the degrees of the vertices in $G_{UZ}(R)$. Considering the adjacency condition between two vertices in $G_{UZ}(R)$ related to units and zero divisors in $R$, the relationship between the degrees of the vertices and the unit group structure in $R$ is determined. The first result shows that the degree of the vertices of $G_{UZ}(R)$ is at most the number of unit elements in $R$.

\begin{lemma} \label{lemma about maximum degree in general}
The maximum degree of the vertices in $G_{UZ}(R)$ is $|U(R)|$.
\end{lemma}
\begin{proof}
Let $0\in R$. It is clear that $0+u=u \in U(R)$ and $0.u=0 \in Z(R)$ for every $u\in U(R)$. We obtain $\deg_{G_{UZ}(R)}(0)=|U(R)|$. Let $x \in R$ be any vertex and $N(x)$ be the set of all vertices in $G_{UZ}(R)$ that are adjacent to $x$. This means that for every $y\in N(x)$, $x+y \in U(R)$ and $xy \in Z(R)$. Note that $R + x = R$; consequently, for every $u \in U(R)$ there exists an element $x_u \in R$ such that $x_u + x = u$. This implies that $x_u$ is uniquely determined by $u$ such that $\deg_{G_{UZ}(R)}(x)\leq|U(R)|$. Thus, there are no other vertices in $G_{UZ}(R)$ whose degree is greater than $\deg_{G_{UZ}(R)}(0)=|U(R)|$. Therefore, the maximum vertex degree in $G_{UZ}(R)$ is $|U(R)|$.
\end{proof}

The maximum degree appears to be the same for all rings. This raises the question of whether the same holds for the minimum degree. Consider the following unit-zero divisor graphs of two different rings.

\begin{figure}[htt]
    \centering
  \begin{minipage}[b]{0.48\textwidth}
    \centering
            \begin{tikzpicture}
             \node[circle, draw, scale=0.8] (0) at (-0.7,1.2) {$\bar{0}$};
             \node[circle, draw, scale=0.8] (1) at (0.7,1.2) {$\bar{1}$};
             \node[circle, draw, scale=0.8] (2) at (-0.7,-1.2) {$\bar{2}$};
             \node[circle, draw, scale=0.8] (3) at (0.7,-1.2) {$\bar{3}$};
             \node[circle, draw, scale=0.8] (4) at (1.6,0) {$\bar{4}$};
             \node[circle, draw, scale=0.8] (5) at (-1.6,0) {$\bar{5}$};
            \Edge(0)(1) \Edge(0)(5)
            \Edge(1)(4) \Edge(2)(3)
            \Edge(2)(5) \Edge(3)(4)
            \end{tikzpicture}
            \captionof{figure}{$G_{UZ}(\mathbb{Z}_{6})$} \label{GambarGuz(Z6)}
\end{minipage}\hfill
\begin{minipage}[b]{0.48\textwidth}
    \centering
            \begin{tikzpicture}
             \node[circle, draw, scale=0.8] (0) at (3.5,0) {$\bar{0}$};
             \node[circle, draw, scale=0.8] (1) at (5,0) {$\bar{1}$};
             \node[circle, draw, scale=0.8] (2) at (6.5,0) {$\bar{2}$};
            \Edge(0)(1) \Edge(1)(2)
            \end{tikzpicture}
            \vspace{1.2cm}
            \captionof{figure}{$G_{UZ}(\mathbb{Z}_{3})$}
            \label{GambarGuz(Z3)}
\end{minipage}
\end{figure}

From Figure \ref{GambarGuz(Z6)} and Figure \ref{GambarGuz(Z3)}, we know that $G_{UZ}(\mathbb{Z}_{3})$ is not a regular graph like the graph $G_{UZ}(\mathbb{Z}_{6})$, hence its minimum degree differs from its maximum degree. Based on this condition, we provide a sufficient condition for the regularity of the unit-zero divisor graph. Before stating it, we establish a preliminary lemma that will be instrumental in proving the regularity of the graph under certain algebraic conditions.

\begin{lemma}\label{jumlahan 2 unit bukan unit}
    If $2 \notin U(R)$, then $u_1+u_2 \notin U(R)$ for all $u_1,u_2 \in U(R)$.
\end{lemma} 
\begin{proof}
    Since $2 \notin U(R)$, $2$ is contained in a maximal ideal of $R$. Let $M$ be a maximal ideal of $R$ such that $2 \in M$. Then, the residue field $R/M$ has characteristic $2$. In particular, if $R/M$ is a finite field of order $2$, then $R/M \cong \mathbb{Z}_2$. Therefore, we can define a ring homomorphism $ f: R \rightarrow \mathbb{Z}_2 $ by
    \[
    f(r) =
    \begin{cases}
        0, & \text{if } r \in M, \\
        1, & \text{if } r \notin M,
    \end{cases}
    \]
    for each $ r \in R $. Let $ u_1, u_2 \in U(R) $. Then $ u_1, u_2 \notin M $, hence $ f(u_1) = 1 $ and $ f(u_2) = 1 $. It follows that $f(u_1 + u_2) = f(u_1) + f(u_2) = 1 + 1 = 0$. This implies $ u_1 + u_2 \in M $. Therefore, $ u_1 + u_2 \notin U(R) $.
\end{proof}

The above lemma shows that the sum of any two units is never a unit under the assumption that $2 \notin U(R)$. This observation plays a key role in analyzing the neighborhood of each vertex in $G_{UZ}(R)$, and leads directly to the following structural result.

\begin{theorem}\label{theorem if 2 is non-unit then U(R)-regular}
    If $2 \notin U(R)$, then $G_{UZ}(R)$ is a $|U(R)|-$regular graph. 
\end{theorem}
\begin{proof}
    Let $x \in R$, consequently $R + x = R$. Hence, for every $u \in U(R)$, there exists an element $x_u \in R$ such that $x_u + x = u$. Thus, $x_u$ is uniquely determined by $u$. Since $2 \notin U(R)$, it follows that $x_u \neq x$. There are two possibilities for $x$, namely $x \in U(R)$ or $x \notin U(R)$. If $x\notin U(R)$, we obtain $x\in Z(R)$ hence $xx_u\in Z(R)$. Therefore, $x_u$ is adjacent to $x$ in the graph $G_{UZ}(R)$. By Lemma \ref{jumlahan 2 unit bukan unit}, if $x \in U(R)$, then $x_u \notin U(R)$. This implies $x_u \in Z(R)$. Therefore, $xx_u \in Z(R)$ and consequently, $x_u$ is adjacent to $x$ in the graph $G_{UZ}(R)$. Let $N(x)$ be the set of all vertices adjacent to $x$. We define a mapping $f: U(R) \rightarrow N(x)$ by $f(u) = x_u$ for each $u\in U(R)$. Since each $x_u$ is uniquely determined by $u$, $f$ is bijective. Thus, $\text{deg}_{G_{UZ}(R)}(x) = |N(x)| = |U(R)|$. Therefore, $G_{UZ}(R)$ is a $|U(R)|$-regular graph.
\end{proof}

This theorem establishes that when the element $2$ is not a unit in $R$, each vertex in $G_{UZ}(R)$ has the same degree, equal to the number of units in $R$. This observation leads to an interesting consequence regarding the Eulerian property of the graph.

\begin{corollary}
    If $2 \notin U(R)$ and $|U(R)|$ are even, then $G_{UZ}(R)$ is a Eulerian graph. 
\end{corollary}
\begin{proof}
    Since $R$ is a finite commutative ring with an identity element and $2\notin U(R)$, by Theorem \ref{theorem if 2 is non-unit then U(R)-regular}, $G_{UZ}(R)$ is a $|U(R)|$-regular graph. Since $|U(R)|$ is even, all vertices in $G_{UZ}(R)$ have even degree, thus $G_{UZ}(R)$ is an Eulerian graph.
\end{proof}

The above result follows that a graph is Eulerian if and only if all vertices have an even degree. It provides an elegant criterion for determining the Eulerian property of $G_{UZ}(R)$ based purely on the structure of the unit group.
Now consider the ring $\mathbb{Z}_9$ which is a local ring with the unique maximal ideal $\langle \bar{3} \rangle = \{ \bar{0}, \bar{3}, \bar{6} \}$. This maximal ideal also coincides with the set of zero-divisors in $\mathbb{Z}_9$. By analyzing the adjacency relations among the elements of the ring, we obtain that the unit-zero divisor graph $G_{UZ}(\mathbb{Z}_9)$ forms a complete bipartite graph. The bipartition reflects the structure of $\mathbb{Z}_9$ as a local ring with a unique maximal ideal, as illustrated in Figure~\ref{GambarGuz(Z9)}.

  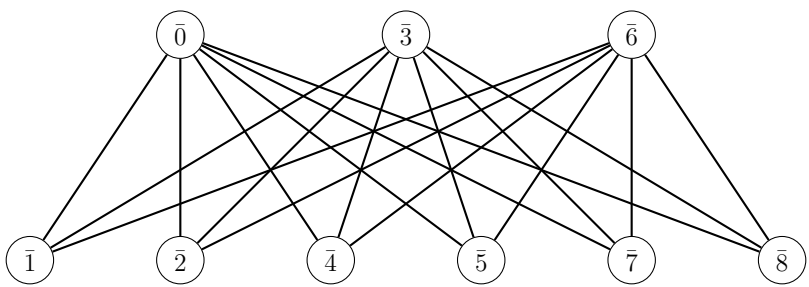
\begin{figure}[htt]
    \centering
    \begin{tikzpicture}
     \node[circle, draw, scale=0.8] (0) at (-3,2) {$\bar{0}$};
     \node[circle, draw, scale=0.8] (1) at (-5,-1) {$\bar{1}$};
     \node[circle, draw, scale=0.8] (2) at (-3,-1) {$\bar{2}$};
     \node[circle, draw, scale=0.8] (3) at (0,2) {$\bar{3}$};
     \node[circle, draw, scale=0.8] (4) at (-1,-1) {$\bar{4}$};
     \node[circle, draw, scale=0.8] (5) at (1,-1) {$\bar{5}$};
     \node[circle, draw, scale=0.8] (6) at (3,2) {$\bar{6}$};
     \node[circle, draw, scale=0.8] (7) at (3,-1) {$\bar{7}$};
     \node[circle, draw, scale=0.8] (8) at (5,-1) {$\bar{8}$};

    \Edge(0)(1) \Edge(0)(2) \Edge(0)(4) \Edge(0)(5)
    \Edge(0)(7) \Edge(0)(8) 
    \Edge(3)(1) \Edge(3)(2) \Edge(3)(4) \Edge(3)(5)
    \Edge(3)(7) \Edge(3)(8) 
    \Edge(6)(1) \Edge(6)(2) \Edge(6)(4) \Edge(6)(5)
    \Edge(6)(7) \Edge(6)(8) 
    \end{tikzpicture}
    \caption{$G_{UZ}(\mathbb{Z}_{9})$} \label{GambarGuz(Z9)}
    \end{figure}

In the following theorem, we give conditions under which $G_{UZ}(R)$ exhibits a bipartite structure, in general.
\begin{theorem}\label{theorem local implies complete bipartite}
    If $R$ is a local ring with $\mathfrak{m}$ as its maximal ideal, then $G_{UZ}(R)$ is a complete bipartite graph $K_{m,n}$, where $m=|\mathfrak{m}|$ and $n=|U(R)|$.
\end{theorem}
\begin{proof}
    Since $R$ is finite, every element in $R\backslash Z(R)$ is a unit. This implies that $R$ forms the partition $\{Z(R),U(R)\}$. In addition, since $R$ is a local ring, its unique maximal ideal is $\mathfrak{m}=Z(R)$.

    Let $u$ and $v$ be two distinct vertices in $G_{UZ}(R)$. If $u,v \in Z(R)$, then $u+v\in Z(R)$ as $Z(R)=\mathfrak{m}$ is an ideal. Thus, $u+v$ is not a unit. Consequently, $u$ and $v$ are not adjacent. If $u,v \in U(R)$, then $uv\in U(R)$. Thus, $uv$ is not a zero divisor. This implies that $u$ and $v$ are not adjacent.
    
    Next, suppose that $u\in U(R)$ and $v\in Z(R)$. Since $u\notin Z(R)$ and $Z(R)$ is a unique maximal ideal, then $u+v \notin Z(R)$. Thus, $u+v\in U(R)$. Since $v\in Z(R)$ and $Z(R)$ is an ideal, we have $uv\in Z(R)$. Hence,  $u+v\in U(R)$ and $uv\in Z(R)$ for every $u\in U(R)$ and $v\in Z(R)$. Therefore, the graph $ G_{UZ}(R) $ is a complete bipartite graph $ K_{m,n} $ where $ m = |Z(R)| = |\mathfrak{m}| $ and $ n = |U(R)|$.
\end{proof}

This result highlights the fundamental bipartite property of $G_{UZ}(R)$ in the case of local rings, where the set of units and the maximal ideal naturally form two independent sets. The converse of this result is also true, as demonstrated in the following theorem.

\begin{theorem}\label{theorem if complete bipartite then local}
    If $G_{UZ}(R)$ is a complete bipartite graph, then $R$ is local.
\end{theorem}
\begin{proof}
    Let $G_{UZ}(R)$ be a complete bipartite graph $K_{m,n}$ with bipartition $\{M, N\}$.
     Let $x,y\in U(R)$, we have $xy\in U(R)$. Hence $xy \notin Z(R)$, which implies that $x$ and $y$ are not adjacent in $G_{UZ}(R)$. Without loss of generality, assume $U(R) \subseteq M$. As the ring $R$ is finite, every element in $R$ is either a unit or a zero divisor such that $R \backslash U(R) = Z(R)$. This implies $N\subseteq Z(R)$. Assume that there exists $m \in M \backslash U(R)$, which means $m \in Z(R)$ and $m$ must be adjacent to every element in $N$. Suppose that $n \in N$, then $m$ and $n$ are adjacent, which means $m+n \in U(R)$ and $mn \in Z(R)$. Since $m$ and $n$ are both zero divisors, $m$ and $n$ do not lie in the same ideal. Hence, $R$ has more than one ideal. Let $I_1$ and $I_2$ be ideals in $R$, then $0 \in I_1 \cap I_2$. For every $ u \in U(R) \subseteq M $, we have $ 0 + u = u \in U(R) $ and $ 0u = 0 \in Z(R) $. This implies that $ 0 $ is adjacent to $ u $. Hence $0$ cannot possibly lie in the partition $M$, which means $0 \in N$ and $0$ must be adjacent to $m \in M \backslash U(R)$. However, $0+m=m\in Z(R)$, leading to a contradiction. Therefore, $m \in M \backslash U(R)$ is impossible to exist. Thus, $M=U(R)$.
    
    Since $R$ is finite, we have $R\backslash U(R)=Z(R)$. As the graph $G_{UZ}(R)$ is a complete bipartite graph, the second set in the vertex partition is $N=Z(R)$. Let $a,b \in Z(R)$, then $a$ and $b$ are not adjacent in $G_{UZ}(R)$. However, since $ab \in Z(R)$, it follows that $a+b \notin U(R)$ or, in other words, $a+b \in Z(R)$. This implies that $-a\in Z(R)$ for every $a\in Z(R)$. Hence, $Z(R)$ is a subring of $R$. Due to $rz \in Z(R)$ for $r\in R$ and $z\in Z(R)$, $Z(R)$ is an ideal. Since $Z(R)=R\backslash U(R)$, there exists no ideal that properly contains $Z(R)$. Thus, $Z(R)$ is the unique maximal ideal in $R$. Therefore, $R$ is a local ring.
\end{proof}

The above theorem confirms that a complete bipartite structure in $G_{UZ}(R)$ strongly indicates that the underlying ring is local. A fascinating special case occurs when $G_{UZ}(R)$ exhibits the structure of a star graph.

\begin{corollary}\label{star iff R field}
    For every ring $R$, $G_{UZ}(R)$ is a star graph if and only if $R$ is a field. 
\end{corollary}
\begin{proof}
    Suppose that $R$ is a field. Thus, $R$ is a local ring with $\{0\}$ as the maximal ideal, and all elements of $R\backslash \{0\}$ are units. Let $|R|=n$, by Theorem \ref{theorem local implies complete bipartite}, $G_{UZ}(R)$ is a complete bipartite graph $K_{1,(n-1)}$, or in other words, it is a star graph.
    
    Conversely, suppose that $G_{UZ}(R)$ is a star graph, then $G_{UZ}(R)$ can be viewed as a complete bipartite graph $K_{1,(n-1)}$ where $n=|R|$. By Theorem \ref{theorem if complete bipartite then local}, if $G_{UZ}(R)$ is a complete bipartite graph, then $R$ is a finite local commutative ring with identity and $Z(R)$ as its maximal ideal. Since all elements besides zero divisors in a finite ring are units, it can be assumed that $R$ is partitioned into $\{U(R), Z(R)\}$. There are two possibilities: first, $|Z(R)|=1$ and $|U(R)|=n-1$, which are satisfied if and only if $R$ is a field. The second possibility is that $|U(R)|=1$ and $|Z(R)|=n-1$, which is satisfied if and only if $R\cong \mathbb{Z}_2$. Therefore, $R$ must be a field.
\end{proof}

It follows that in a field, every non-zero element is a unit and the only zero divisor is the zero element, resulting in a starlike structure for the unit-zero divisor graph. This observation highlights the strong influence of the ring’s algebraic properties on the topology of $G_{UZ}(R)$. Building upon this foundational case, we next turn our attention to more intricate structural features of the graph. In particular, we investigate conditions under which $G_{UZ}(R)$ exhibits Eulerian or Hamiltonian properties, revealing deeper combinatorial characteristics tied to the ring's structure.

\begin{theorem}\label{eulerian condition if R local}
    Let $R$ be a finite local commutative ring with identity. Then $G_{UZ}(R)$ is an Eulerian graph if and only if both $|R|$ and $|U(R)|$ are even.
\end{theorem}
\begin{proof}
    By Theorem \ref{theorem local implies complete bipartite}, if $R$ is a finite local commutative ring with an identity element,  $G_{UZ}(R)$ is a complete bipartite graph with bipartition $\{U(R), Z(R)\}$. Suppose that $|R|$ is even and $|U(R)|$ is even, thus $|Z(R)|=|R|-|U(R)|$ is also even. This implies that every vertex in $G_{UZ}(R)$ has an even degree. By Euler's Theorem, we can conclude that $G_{UZ}(R)$ is Eulerian. 
    
    Conversely, suppose that the graph $G_{UZ}(R)$ is Eulerian, hence each of its vertices must have an even degree. Because the graph $G_{UZ}(R)$ is a complete bipartite graph, the number of vertices in each subset of the partition, in this case $U(R)$ and $Z(R)$, must be even. It shows that $|U(R)|$ is even and $|R|=|U(R)|+|Z(R)|$ is also even.
\end{proof}

\begin{theorem}\label{hamiltonian condition if R local}
    Let $R$ be a finite local commutative ring with identity. Then $G_{UZ}(R)$ is a Hamiltonian graph if and only if $|U(R)|=|Z(R)|$.
\end{theorem}
\begin{proof}
    By Theorem \ref{theorem local implies complete bipartite}, if $R$ is a finite local commutative ring with an identity element,  $G_{UZ}(R)$ is a complete bipartite graph with bipartition $\{U(R), Z(R)\}$. A complete bipartite graph is Hamiltonian if and only if the number of vertices in both partitions must be the same. Thus, the graph $G_{UZ}(R)$ is a Hamiltonian graph if and only if $|U(R)|=|Z(R)|$.
\end{proof}

Next, we examine the planarity of $G_{UZ}(R)$ in the case where $R$ is a finite local ring. Planarity, as a fundamental topological property, provides insight into how densely connected the graph is. The following theorem characterizes precisely when $G_{UZ}(R)$ can be drawn in the plane without edge crossings in terms of the number of units and zero divisors in $R$.

\begin{theorem}\label{planarity condition if R local}
    Let $R$ be a finite local commutative ring with identity. Then $G_{UZ}(R)$ is a planar graph if and only if $|U(R)|\leq 2$ or $|Z(R)|\leq 2$.
\end{theorem}
\begin{proof}
    By Theorem \ref{theorem local implies complete bipartite}, if $R$ is a finite local commutative ring with an identity element, then $G_{UZ}(R)$ is a complete bipartite graph $K_{m,n}$ with $m=|Z(R)|$ and $n=|U(R)|$. This implies that the graph $G_{UZ}(R)$ does not contain a subgraph that is a subdivision of $K_{3,3}$ if and only if $|U(R)|\leq 2$ or $|Z(R)|\leq 2$. By Kuratowski's Theorem, we conclude that $G_{UZ}(R)$ is a planar graph if and only if $|U(R)|\leq 2$ or $|Z(R)|\leq 2$.
\end{proof}

Finally, we summarize several structural parameters of $G_{UZ}(R)$ in the local ring case.

\begin{corollary}
    If $R$ is a local ring, then
    \begin{enumerate}[label=(\roman*)]
        \item diameter of $G_{UZ}(R)$ is 2,
        \item girth of $G_{UZ}(R)$ is 4,
        \item clique number of $G_{UZ}(R)$ is 2,
        \item chromatic number of $G_{UZ}(R)$ is 2,
        \item independence number of $G_{UZ}(R)$ is $max(|U(R)|,|Z(R)|)$,
        \item domination number of $G_{UZ}(R)$ is 2.
    \end{enumerate}
\end{corollary}
\begin{proof}
    By Theorem \ref{theorem local implies complete bipartite}, if $R$ is a local ring, then $G_{UZ}(R)$ is a complete bipartite graph. Consequently, the values of these graph parameters follow directly.
\end{proof}

Previously, we established that if $R$ is a local ring, then $G_{UZ}(R)$ forms a complete bipartite graph. A natural extension of this result is to ask whether weaker conditions might still guarantee that $G_{UZ}(R)$ is bipartite, even if not necessarily complete. The following theorem provides conditions under which the bipartite nature of the graph is preserved.

\begin{theorem}\label{bipartite condition}
    If $R$ is a local ring or $2\notin U(R)$, then $G_{UZ}(R)$ is a bipartite graph.
\end{theorem}
\begin{proof}
    Firstly, suppose that $R$ is a finite local commutative ring with an identity element. By Theorem \ref{theorem local implies complete bipartite}, $G_{UZ}(R)$ is a bipartite graph.

    Secondly, since $2 \notin U(R)$, $2$ is contained in a maximal ideal of $R$. Let $M$ be a maximal ideal of $R$ such that $2 \in M$. Then, the residue field $R/M$ has characteristic $2$. In particular, if $R/M$ is a finite field of order $2$, then $R/M \cong \mathbb{Z}_2$. Therefore, we can define a homomorphism $f:R \rightarrow \mathbb{Z}_2$ where
    \[
    f(r) =
    \begin{cases}
    0, & \text{if } r \in M, \\
    1,  & \text{if } r \notin M
    \end{cases}
    \]
    for every $r\in R$. Assume that the vertices of $G_{UZ}(R)$ are partitioned into $M$ and $R\backslash M$. Next, we will show that two vertices within the same partition are not adjacent to each other.
    Let $a_1,a_2 \in M$, it is clear that $a_1+a_2\in M$ so that $a_1+a_2$ is a zero divisor. Consequently, $a_1$ and $a_2$ are not adjacent.
    Let  $r_1,r_2 \in R\backslash M$, we have $f(r_1 + r_2)=f(r_1)+f(r_2)=1+1=0$ so that $r_1 + r_2 \in M$. That means $r_1 + r_2$ is not a unit, which implies that $r_1$ is not adjacent to $r_2$.
    Thus, $G_{UZ}(R)$ is a bipartite graph.
\end{proof}

Extending the previous result, we investigate a more general classification of the graph based on the number of maximal ideals of ring $R$. Consider the ring $\mathbb{Z}_{15} $. The set of units in $ \mathbb{Z}_{15} $ is  
$
U(\mathbb{Z}_{15}) = \{\bar{1}, \bar{2}, \bar{4}, \bar{7}, \bar{8}, \bar{11}, \bar{13}, \bar{14}\}.
$  
The ring $ \mathbb{Z}_{15} $ has two distinct maximal ideals, namely $ \langle \bar{3} \rangle = \{\bar{0}, \bar{3}, \bar{6}, \bar{9}, \bar{12} \} $ and $ \langle \bar{5} \rangle = \{\bar{0}, \bar{5}, \bar{10} \} $. By observation of the adjacency among the vertices, we obtain that the unit-zero divisor graph $ G_{UZ}(\mathbb{Z}_{15}) $ forms a $3$-partite graph. The vertex set of the graph can be partitioned into three disjoint subsets $\{ U(\mathbb{Z}_{15}),\ \langle \bar{3} \rangle,\ \langle \bar{5} \rangle \setminus \{\bar{0}\} \}$. This observation demonstrates that the number of maximal ideals in a ring has a significant influence on the partitioning of the unit-zero divisor graph associated with the ring, as illustrated in Figure~\ref{GambarGuz(Z15)}.

\begin{figure}[htt]
    \centering
    \begin{tikzpicture}
     \node[circle, draw, scale=0.8] (0) at (-6,3) {$\bar{0}$};
     \node[circle, draw, scale=0.8] (1) at (-4.5,0) {$\bar{1}$};
     \node[circle, draw, scale=0.8] (2) at (-3,0) {$\bar{2}$};
     \node[circle, draw, scale=0.8] (3) at (-3,3) {$\bar{3}$};
     \node[circle, draw, scale=0.8] (4) at (-1.5,0) {$\bar{4}$};
     \node[circle, draw, scale=0.8] (5) at (-4,-3) {$\bar{5}$};
     \node[circle, draw, scale=0.8] (6) at (0,3) {$\bar{6}$};
     \node[circle, draw, scale=0.8] (7) at (0,0) {$\bar{7}$};
     \node[circle, draw, scale=0.8] (8) at (1.5,0) {$\bar{8}$};
     \node[circle, draw, scale=0.8] (9) at (3,3) {$\bar{9}$};
     \node[circle, draw, scale=0.8] (10) at (4,-3) {$\bar{10}$};
     \node[circle, draw, scale=0.8] (11) at (3,0) {$\bar{11}$};
     \node[circle, draw, scale=0.8] (12) at (6,3) {$\bar{12}$};
     \node[circle, draw, scale=0.8] (13) at (4.5,0) {$\bar{13}$};
     \node[circle, draw, scale=0.8] (14) at (6,0) {$\bar{14}$};

    \Edge(0)(1) \Edge(0)(2) \Edge(0)(4) \Edge(0)(7)
    \Edge(0)(8) \Edge(0)(11) \Edge(0)(13) \Edge(0)(14)
    \Edge(1)(3) \Edge(1)(6) \Edge(1)(10) \Edge(1)(12)
    \Edge(2)(5) \Edge(2)(6) \Edge(2)(9) \Edge(2)(12)
    \Edge(3)(13) \Edge(3)(14) \Edge(3)(1) \Edge(3)(4)
    \Edge(3)(5) \Edge(3)(8) \Edge(3)(10) \Edge(3)(11)
    \Edge(4)(12) \Edge(4)(9) \Edge(4)(10)
    \Edge(5)(11) \Edge(5)(12) \Edge(5)(14) \Edge(5)(6)
    \Edge(5)(8) \Edge(5)(9)
    \Edge(6)(10) \Edge(6)(11) \Edge(6)(13) \Edge(6)(7) 
    \Edge(6)(8)
    \Edge(7)(9) \Edge(7)(10) \Edge(7)(12)
    \Edge(8)(9) 
    \Edge(9)(10) \Edge(9)(13) \Edge(9)(14)
    \Edge(10)(12) \Edge(10)(13)
    \Edge(11)(12) 
    \Edge(12)(14)
    \end{tikzpicture}
    \caption{$G_{UZ}(\mathbb{Z}_{15})$} \label{GambarGuz(Z15)}
\end{figure}
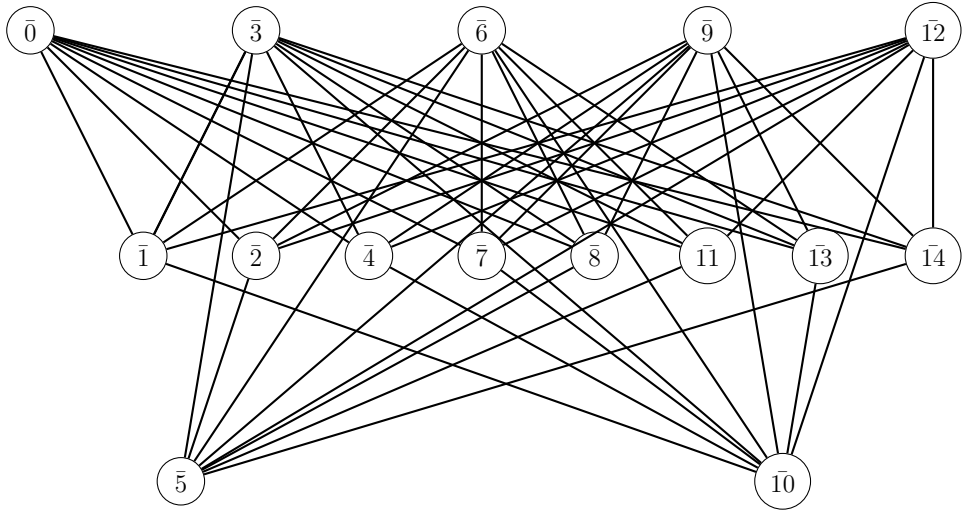

In the following theorem, we give the classification of $G_{UZ}(R)$ based on the number of maximal ideals of the ring $R$. 
\begin{theorem}\label{the condition for Guz is n+1 partite}
    If $R$ has $n$ maximal ideals, then $G_{UZ}(R)$ is a $(n+1)$-partite graph.
\end{theorem}
\begin{proof}
    Let $\mathcal{M}=\{I_i \mid i=\{1,2,..,n\}$ be the set of all maximal ideals in $R$ and form a partition on the vertex set $V(G_{UZ}(R))$, namely $P=\{U(R), I_1, I^*_2, I^*_3, ... , I^*_n\}$ where $ I^*_2=I_2 - I_1, I^*_3=I_3 - (I_1 \cup I_2), ... , I^*_n=I_n- (\bigcup _{i=1}^{n-1} I_i)$. For every $k\leq n$, the set $I^*_k$ is non-empty because each $I_i$ is a maximal ideal. Since $R$ is a finite commutative ring with an identity element, every element in $R$ must be either a unit or a zero divisor. Because every non-unit element in $R$ is contained in a maximal ideal, $U(R)\cup I_1\cup I^*_2\cup I^*_3\cup ... \cup I^*_n=R$ and all sets in the partition are disjoint. Next, we will show that every set in the partition $P$ is independent.
    
    Let $u_1, u_2 \in U(R)$, then we have $u_1 u_2 \in U(R)$, which means that $u_1 u_2$ is not a zero divisor. This implies that $u_1$ and $u_2$ are not adjacent in the graph $G_{UZ}(R)$. Let $a,b\in I_i$ where $i=\{1,2,...,n\}$. Because $I_i$ is an ideal, then $a+b \in I_i$. Since $I_i$ does not contain any units, it follows that $a+b \notin U(R)$. Therefore, $a$ and $b$ are not adjacent in the graph $G_{UZ}(R)$. Thus, for every $i\in\{1,2,...,n\}$, $I^*_i$ is an independent set, proving that the graph $G_{UZ}(R)$ is a $(n+1)$-partite graph.
\end{proof}

Although in general, if a ring $R$ has $n$ maximal ideals, then $G_{UZ}(R)$ forms a $(n+1)$-partite graph. In the specific case where $2 \notin U(R)$ and $R$ is not a local ring, implying that $n > 1$, Theorem \ref{bipartite condition} shows that $G_{UZ}(R)$ can still be regarded as a bipartite graph, particularly in the context of graph coloring. Consider the ring $\mathbb{Z}_6$, which has two maximal ideals: $\langle 2 \rangle = \{0, 2, 4\}$ and $\langle 3 \rangle = \{0, 3\}$. According to Theorem
\ref{the condition for Guz is n+1 partite}, $G_{UZ}(\mathbb{Z}_6)$ can be viewed as a 3-partite graph, as illustrated in Figure \ref{GambarGuz(Z6)-3partit}. However, $G_{UZ}(\mathbb{Z}_6)$ can also be interpreted as a bipartite graph, as shown in Figure \ref{GambarGuz(Z6)bipartit}.

\begin{figure}[htt]
\centering
    \begin{minipage}[b]{0.48\textwidth}
    \centering
            \begin{tikzpicture}
             \node[circle, draw, scale=0.8] (0) at (-1,0) {$\bar{0}$};
             \node[circle, draw, scale=0.8] (1) at (0.6,1.2) {$\bar{1}$};
             \node[circle, draw, scale=0.8] (2) at (0,0) {$\bar{2}$};
             \node[circle, draw, scale=0.8] (3) at (0,-1.2) {$\bar{3}$};
             \node[circle, draw, scale=0.8] (4) at (1,0) {$\bar{4}$};
             \node[circle, draw, scale=0.8] (5) at (-0.6,1.2) {$\bar{5}$};
            \Edge(0)(1) \Edge(0)(5)
            \Edge(1)(4) \Edge(2)(3)
            \Edge(2)(5) \Edge(3)(4)
            \end{tikzpicture}
            \captionof{figure}{$G_{UZ}(\mathbb{Z}_{6})$ as a 3-partite graph} \label{GambarGuz(Z6)-3partit}
    \end{minipage}\hfill
    \begin{minipage}[b]{0.48\textwidth}
    \centering
            \begin{tikzpicture}
             \node[circle, draw, scale=0.8] (0) at (-1,1.2) {$\bar{0}$};
             \node[circle, draw, scale=0.8] (1) at (-1,-1.2) {$\bar{1}$};
             \node[circle, draw, scale=0.8] (2) at (0,1.2) {$\bar{2}$};
             \node[circle, draw, scale=0.8] (3) at (0,-1.2) {$\bar{3}$};
             \node[circle, draw, scale=0.8] (4) at (1,1.2) {$\bar{4}$};
             \node[circle, draw, scale=0.8] (5) at (1,-1.2) {$\bar{5}$};
            \Edge(0)(1) \Edge(0)(5)
            \Edge(1)(4) \Edge(2)(3)
            \Edge(2)(5) \Edge(3)(4)
            \end{tikzpicture}
            \captionof{figure}{$G_{UZ}(\mathbb{Z}_{6})$ as a bipartite graph} \label{GambarGuz(Z6)bipartit}
    \end{minipage}
\end{figure}

An analysis of adjacency relations in $G_{UZ}(R)$ shows that the interaction patterns among vertices in the graph can be categorized based on the ideals of the ring $R$. This provides a more nuanced view of the structure of the graph, emphasizing the interaction between the algebraic properties of ideals and the combinatorial elements of the corresponding graph. The next structural property of the unit-zero divisor graph of a quotient ring arises directly from the fact that the maximal ideal induces a field structure on the quotient. This property is established in Lemma~\ref{lemma if m maks ideal maks then R/m is star}.

\begin{lemma}\label{lemma if m maks ideal maks then R/m is star}
    If $\mathfrak{m}$ is a maximal ideal of $R$, then $G_{UZ}(R/\mathfrak{m})$ is a star graph. 
\end{lemma}
\begin{proof}
    Since $\mathfrak{m}$ is a maximal ideal, then $R/\mathfrak{m}$ is a field. By Theorem \ref{star iff R field}, we have $G_{UZ}(R/\mathfrak{m})$ is a star graph.
\end{proof}

The next result further explores the structure of the unit-zero divisor graph of a finite commutative ring by examining the internal adjacency relations within the cosets of a proper ideal. In particular, we show that each coset forms an independent set in the graph $G_{UZ}(R)$ when $2 + I$ is not a unit in the quotient ring $R/I$.

\begin{theorem}\label{theorem every coset is independent if 2 is not a unit}
Let $I$ be a proper ideal of $R$. If $2 + I$ is not a unit in the quotient ring $R/I$, then the coset $r + I$ forms an independent set in the unit-zero divisor graph $G_{UZ}(R)$ for any $r \in R$.
\end{theorem}
\begin{proof}
    Suppose that $I$ is a proper ideal in $R$ and $2+I$ is not a unit in $R/I$. Let $x,y\in r+I$ be two distinct elements in the coset $r+I\in R/I$. This means $x=r+i_x$ and $y=r+i_y$ for some $i_x,i_y \in I$. Thus, there exists $a \in I$ such that $x = y + a$. We have  $x + y = (y + a) + y = 2y + a.$
    Since $a \in I$, then $x + y + I = 2y + I$. Observe that $x + y = r+i_x+r+i_y = 2r + i_x+i_y$, hence $x + y + I = 2r+I=(2 + I)(r + I)$ in $R/I$.
    
    Suppose that $x$ and $y$ are adjacent in $G_{UZ}(R)$, which means that $x + y$ is a unit in $R$ and $xy$ is a zero divisor in $R$. The canonical homomorphism of $R \to R/I$ maps units to units. Since $x + y$ is a unit in $R$, then $x + y + I$ is a unit in $R/I$. This implies that $(2 + I)(r+ I)$ is also a unit in $R/I$.
    Since $R/I$ is a finite commutative ring, every element in $R/I$ is either a unit or a zero divisor. If $(2 + I)(r+ I)$ is a unit in $R/I$, then $2 + I$ must also be a unit in $R/I$, a contradiction. Therefore, every coset $r+I$ is an independent set in $G_{UZ}(R)$.
\end{proof}

An interesting behavior emerges when the observation of ideals is restricted specifically to the Jacobson radical $J(R)$ of the ring $R$. Since $J(R)$ is the intersection of all maximal ideals of $R$, it plays a central role in determining the annihilation properties that influence adjacency in $G_{UZ}(R)$. By examining the image of each element in the quotient ring $R/J(R)$, we can better understand how the structure of $G_{UZ}(R)$ simplifies when $R$ is factored by $J(R)$.  Consider the ring $\mathbb{Z}_{12}$. Figure \ref{GambarGuz(Z12)} and Figure \ref{GambarGuz(Z12/J)} illustrate how adjacency is preserved between $G_{UZ}(\mathbb{Z}_{12})$ and $G_{UZ}(\mathbb{Z}_{12}/J(\mathbb{Z}_{12}))$.

\begin{figure}[htt]
\centering
\begin{minipage}[b]{0.48\textwidth}
    \centering
    \begin{tikzpicture}
     \node[circle, draw, scale=0.8] (0) at (-0.5,0.8) {$\bar{0}$};
     \node[circle, draw, scale=0.8] (1) at (0.5,0.8) {$\bar{1}$};
     \node[circle, draw, scale=0.8] (2) at (-0.5,-0.8) {$\bar{2}$};
     \node[circle, draw, scale=0.8] (3) at (0.5,-0.8) {$\bar{3}$};
     \node[circle, draw, scale=0.8] (4) at (1.2,0) {$\bar{4}$};
     \node[circle, draw, scale=0.8] (5) at (-1.2,0) {$\bar{5}$};
     \node[circle, draw, scale=0.8] (6) at (-1,1.5) {$\bar{6}$};
     \node[circle, draw, scale=0.8] (7) at (1,1.5) {$\bar{7}$};
     \node[circle, draw, scale=0.8] (8) at (-1,-1.5) {$\bar{8}$};
     \node[circle, draw, scale=0.8] (9) at (1,-1.5) {$\bar{9}$};
     \node[circle, draw, scale=0.8] (10) at (2.1,0) {$\bar{10}$};
     \node[circle, draw, scale=0.8] (11) at (-2.1,0) {$\bar{11}$};
    \Edge(0)(1) \Edge(0)(5) \Edge(0)(7) \Edge(0)(11)
    \Edge(1)(4) \Edge(1)(6) \Edge(1)(10)
    \Edge(2)(3) \Edge(2)(5) \Edge(2)(9) \Edge(2)(11)
    \Edge(3)(4) \Edge(3)(8) \Edge(3)(10)
    \Edge(4)(7) \Edge(4)(9) \Edge(5)(6) \Edge(5)(8)
    \Edge(6)(7) \Edge(6)(11) \Edge(7)(10)
    \Edge(8)(9) \Edge(8)(11) \Edge(9)(10)
    \end{tikzpicture}
    \caption{$G_{UZ}(\mathbb{Z}_{12})$} \label{GambarGuz(Z12)}
\end{minipage}
\hfill
\begin{minipage}[b]{0.48\textwidth}
   \centering
        \begin{tikzpicture}
         \node[circle, draw, scale=0.8] (0) at (-0.7,1.5) {$\bar{\bar{0}}$};
         \node[circle, draw, scale=0.8] (1) at (0.7,1.5) {$\bar{\bar{1}}$};
         \node[circle, draw, scale=0.8] (2) at (-0.7,-1.5) {$\bar{\bar{2}}$};
         \node[circle, draw, scale=0.8] (3) at (0.7,-1.5) {$\bar{\bar{3}}$};
         \node[circle, draw, scale=0.8] (4) at (2,0) {$\bar{\bar{4}}$};
         \node[circle, draw, scale=0.8] (5) at (-2,0) {$\bar{\bar{5}}$};
        \Edge(0)(1) \Edge(0)(5)
        \Edge(1)(4) \Edge(2)(3)
        \Edge(2)(5) \Edge(3)(4)
        \end{tikzpicture}
        \caption{$G_{UZ}(\mathbb{Z}_{12}/J(\mathbb{Z}_{12}))$} \label{GambarGuz(Z12/J)}
\end{minipage}
\end{figure}

The following theorem establishes that if two cosets of $J(R)$ are adjacent in $G_{UZ}(R/J(R))$, then every element within these cosets remains adjacent in $G_{UZ}(R)$. This highlights the influence of the Jacobson radical in maintaining connectivity within the graph.

\begin{theorem}\label{adjacency in the factor ring of the Jacobson radical}
    Let $J(R)$ be the Jacobson radical of $R$. For any $x,y \in R$, if $x+J(R)$ is adjacent to $y+J(R)$ in $G_{UZ}(R/J(R))$, then every element of $x+J(R)$ is adjacent to every element of $y+J(R)$ in $G_{UZ}(R)$.
\end{theorem}
\begin{proof}
    Suppose that $x+J(R)$ is adjacent to $y+J(R)$ in $G_{UZ}(R/J(R))$, which means that $(x+J(R))+(y+J(R))=x+y+J(R)\in U(R/J(R))$ and $(x+J(R))(y+J(R))=xy+J(R)\in Z(R/J(R))$. 
    
    Let $a,b\in R$ where $a\in x+J(R)$ and $b\in y+J(R)$. We can write $a=x+j_a$ and $b=y+j_b$ where $j_a,j_b \in J(R)$. Since $x+y+J(R)\in U(R/J(R))$, there exists $u\in U(R)$ such that $x+y+J(R)=u+J(R)$. We obtain $x+y-u=(a-j_a)+(b-j_b)-u=(a+b)-(j_a+j_b)-u \in J(R)$, and consequently $a+b$ is a unit in $R$. If we assume $a+b$ is not a unit, then $\langle a+b \rangle$ is a proper ideal in $R$. Hence there exists a maximal ideal in $R$, let's call it $M_1$, such that $\langle a+b \rangle \subseteq M_1$ and $a+b \in M_1$. On the other hand, $j_a,j_b \in M_1$ and $(a+b)-(j_a+j_b)-u \in M_1$ are known. This implies $u \in M_1$, leading to a contradiction because $u$ is a unit in $R$. Therefore, $a+b$ is a unit in $R$.
    
    Since $(x+J(R))(y+J(R))=xy+J(R)\in Z(R/J(R))$, there exists $z\in Z(R)$ such that $xy+J(R)=z+J(R)$. We obtain $ xy-z=(a-j_a)(b-j_b)-z=(ab)+(j_aj_b)-(j_ab+j_ba)-z \in J(R) $. Since $z \in Z(R)$, $z$ is not a unit.  Hence, there exists a maximal ideal in $R$, let it be $M_2$, such that $z \in M_2$. On the other hand, we have $j_aj_b \in M_2$ and $j_a b+j_b a \in M_2$. This implies that $ab$ is not a unit, thus $ab \in Z(R)$. Therefore, $a$ and $b$ are adjacent in $G_{UZ}(R)$.
\end{proof}

Conversely, the following theorem shows that adjacency in $G_{UZ}(R)$ implies adjacency in the graph $G_{UZ}(R/J(R))$. It means that the unit-zero divisor graph of the quotient ring captures the essential adjacency relationships present in the original graph. This result is beneficial for reducing complexity when analyzing the graph structure of larger rings by considering their quotient counterparts.

\begin{theorem}\label{adjacent in R implies adjacent in the factor ring of the Jacobson radical}
    Let $J(R)$ be the Jacobson radical of $R$. For any $x,y \in R$, if $x$ is adjacent to $y$ in $G_{UZ}(R)$, then $x+J(R)$ is adjacent to $y+J(R)$ in $G_{UZ}(R/J(R))$.
\end{theorem}
\begin{proof}
    Suppose that $x$ is adjacent to $y$ in $G_{UZ}(R)$, which means $x+y\in U(R)$ and $xy\in Z(R)$. By canonical homomorphism $ \phi: R \to R/J(R) $, every unit in $R$ is mapped to a unit in $ R/J(R) $, and every zero divisor in $R$ is mapped to zero divisor in $R/J(R)$. This implies $ x+y+J(R) \in U(R/J(R)) $ and $xy+J(R)\in Z(R/J(R))$. Therefore, $ x+J(R) $ is adjacent to $ y+J(R) $ in $ G_{UZ}(R/J(R)) $.
\end{proof}

The previous two theorems establish a fundamental link between the adjacency relations in $G_{UZ}(R)$ and those in its quotient graph $G_{UZ}(R/J(R))$. Specifically, they show that the adjacency in the quotient graph corresponds to adjacency between entire cosets in $G_{UZ}(R)$, and vice versa. This observation suggests that key structural properties of $G_{UZ}(R)$ may be preserved when passing to the quotient ring.

One natural question is whether this preservation extends to graph-theoretic parameters such as the diameter. The following theorem confirms that the diameter of $G_{UZ}(R)$ remains unchanged under the Jacobson radical quotient.

\begin{theorem}
    Let $R\ncong\mathbb{Z}_2$ and $J(R)$ be the Jacobson radical of $R$. Then, the diameter of $G_{UZ}(R)$ is equal to the diameter of $G_{UZ}(R/J(R))$ if and only if both $G_{UZ}(R)$ and $G_{UZ}(R/J(R))$ are connected and $R/J(R)\ncong\mathbb{Z}_2$.
\end{theorem}
\begin{proof}
    ($\Rightarrow$)We will prove this part by contraposition, namely if $G_{UZ}(R)$ or $G_{UZ}(R/J(R))$ is not connected or $R/J(R)\cong \mathbb{Z}_2$, then the diameter of $G_{UZ}(R)$ is not equal to the diameter of $G_{UZ}(R/J(R))$. If $R/J(R) \cong \mathbb{Z}_2$, then $2+J(R)=0+J(R)$. Thus, $2+J(R)$ is not a unit in $R/J(R)$. By Theorem \ref{theorem every coset is independent if 2 is not a unit}, the coset $r+I$ is an independent set in $G_{UZ}(R)$ for every $r\in R$. Hence $0+J(R)$ and $1+J(R)$ are independent sets in $G_{UZ}(R)$. As $ R \not\cong \mathbb{Z}_2 $, we have $ |0+J(R)| \geq 2 $. Since $0+J(R)$ is an independent set, we have $d(a,b)>1$ for every $a,b\in 0+J(R)$, thus $\text{diam}(G_{UZ}(R))>1$. However, $\text{diam}(G_{UZ}(R/J(R)))=1$ because $R/J(R)\cong \mathbb{Z}_2$. Therefore, the diameter of $G_{UZ}(R)$ is not equal to the diameter of $G_{UZ}(R/J(R))$.
    
    ($\Leftarrow$) Let $G_{UZ}(R)$ and $G_{UZ}(R/J(R))$ be connected and $R/J(R) \ncong \mathbb{Z}_2$. First, we show that $\text{diam}\big(G_{UZ}(R)\big) \leq \text{diam}\big(G_{UZ}(R/J(R))\big)$.
    Let the diameter of $G_{UZ}(R/J(R))$ be $m$, which means that any two distinct cosets $x + J(R)$ and $y + J(R)$ represented by two vertices in $G_{UZ}(R/J(R))$ can be connected by the shortest path with a maximum length of $m$. 
    Let $a, b$ be the two distinct vertices in $G_{UZ}(R)$. We divide it into two cases:
    \begin{enumerate}[label=(\roman*)]
    \item If $a+J(R)\neq b+J(R)$, since the diameter of $G_{UZ}(R/J(R))$ is $m$, we have $d(a+J(R),b+J(R))\leq m$. Let the shortest path between the two cosets $a + J(R)$ and $b + J(R)$ in $G_{UZ}(R/J(R))$ be
    \[
    a + J(R) = a_1 + J(R) \rightarrow a_2 + J(R) \rightarrow \dots \rightarrow a_{n+1} + J(R) = b + J(R)
    \] 
    of length $n\leq m$.
    By Theorem \ref{adjacency in the factor ring of the Jacobson radical}, if the coset $x+J(R)$ is adjacent to the coset $y+J(R)$ in $G_{UZ}(R/J(R))$, then every element in the coset $x + J(R)$ is also adjacent to every element in the coset $y + J(R)$ in $G_{UZ}(R)$.
    We take $a_i \in a_i + J(R)$ for $2\leq i \leq n$, such that $a_i$ is adjacent to $a_{i+1}$ in $G_{UZ}(R)$. Thus, we have a path
    \[
    a \rightarrow a_2 \rightarrow a_3 \rightarrow \dots \rightarrow a_n \rightarrow b
    \]
    in $G_{UZ}(R)$. Since this path has a length of $n\leq m$, the diameter of $G_{UZ}(R) \leq n \leq m$. Therefore $\text{diam}\big(G_{UZ}(R)\big) \leq \text{diam}\big(G_{UZ}(R/J(R))\big)$.
    
    \item If $a+J(R)=b+J(R)$, since $G_{UZ}(R/J(R))$ is connected, then $a+J(R)$ is not an isolated vertex. Hence, there exists $r+J(R)$ that is adjacent to $a+J(R)$. By Theorem \ref{adjacency in the factor ring of the Jacobson radical}, if the coset $x+J(R)$ is adjacent to the coset $y+J(R)$ in $G_{UZ}(R/J(R))$, then every element in the coset $x + J(R)$ is also adjacent to every element in the coset $y + J(R)$ in $G_{UZ}(R)$. As a result, in this case, there exists $r\in r+J(R)$ such that a path $a\rightarrow r\rightarrow b$ is formed, which means $d(a,b)=2$. Since $R/J(R)\ncong \mathbb{Z}_2$, it follows that $\text{diam}(G_{UZ}(R/J(R)))\geq 2$. Therefore $\text{diam}\big(G_{UZ}(R)\big) \leq \text{diam}\big(G_{UZ}(R/J(R))\big)$.
    \end{enumerate}

    Next, we show that $\text{diam}\big(G_{UZ}(R/J(R))\big) \leq \text{diam}\big(G_{UZ}(R)\big)$. 
    Let the diameter of $G_{UZ}(R)$ be $k$, which means that there is a path with a maximum length of $k$ for every two vertices $x, y \in R$.
    Let $x+J(R)$ and $y+J(R)$ be two vertices in $G_{UZ}(R/J(R))$. Since the diameter of $G_{UZ}(R)$ is $k$, then $d(x,y)\leq k$ in the graph $G_{UZ}(R)$. Let the shortest path between vertices $x$ and $y$ be
    \[
    x=a_1 \rightarrow a_2 \rightarrow a_3 \rightarrow \dots \rightarrow a_{m+1} = y
    \]
    of length $k$.
    By Theorem \ref{adjacent in R implies adjacent in the factor ring of the Jacobson radical}, if $x$ is adjacent to $y$ in $G_{UZ}(R)$, then $x+J(R)$ is adjacent to $y+J(R)$ in $G_{UZ}(R/J(R))$.
    Thus, we obtain a walk from $x+J(R)$ to $y+J(R)$ in $G_{UZ}(R)$ as follows
     \[
    x + J(R) = a_1 + J(R) \rightarrow a_2 + J(R)  \rightarrow \dots \rightarrow a_{m+1} + J(R) = y + J(R).
     \]
    Since the path connecting the two cosets $x+J(R)$ and $y+J(R)$ has a length of $m\leq k$, $\text{diam}(G_{UZ}(R/J(R))) \leq m \leq \text{diam}(G_{UZ}(R))$.
    
    By the two sections above, we conclude that the diameter of $G_{UZ}(R)$ is equal to the diameter of $G_{UZ}(R/J(R))$.
\end{proof}

\section{The unit-zero divisor graph of the ring of integers modulo}

In this section, we focus on the structural properties of the unit-zero divisor graph associated with the ring of integers modulo $n$, denoted by $\mathbb{Z}_n$. Since $\mathbb{Z}_n$ is a well-understood and fundamental example of a finite commutative ring, it serves as a natural and instructive setting for analyzing the behavior of $G_{UZ}(\mathbb{Z}_n)$ under various arithmetic conditions on $n$.
We begin with the case where $G_{UZ}(\mathbb{Z}_n)$ exhibits a star-like structure. This case occurs precisely when $\mathbb{Z}_n$ is a field, that is when $n$ is prime. The following remark summarizes this observation.

\begin{remark}\label{remark Zn star iff n prime}
    The unit-zero graph $G_{UZ}(\mathbb{Z}_n)$ is a star graph if and only if $n$ is prime.
\end{remark}

\noindent
This follows directly from the fact that $\mathbb{Z}_n$ is a field if and only if $n$ is prime, and by Corollary~\ref{star iff R field}, which characterizes when $G_{UZ}(R)$ takes the form of a star graph for a general ring $R$.

Having established when the graph is a star, we next investigate the bipartite and complete bipartite properties of $G_{UZ}(\mathbb{Z}_n)$ under different conditions on $n$. In particular, we examine the role of parity and prime power factorizations of $n$ in determining whether the graph is bipartite or even a complete bipartite graph.

\begin{theorem}\label{theorem Zn bipartite iff n even}
    If $n$ is even, then $G_{UZ}(\mathbb{Z}_n)$ is a bipartite graph.
\end{theorem}
\begin{proof}
If $n$ is even, then $2\in \mathbb{Z}_n$ is a zero divisor. Thus, $2$ is not a unit. By Theorem \ref{bipartite condition}, $G_{UZ}(\mathbb{Z}_n)$ is a bipartite graph if $n$ is even.
\end{proof}

\begin{theorem}\label{theorem Zn complete bipartite if n=p^k}
    If $ n=p^k $ for a prime number $p$ and a positive integer $k$, then $ G_{UZ}(\mathbb{Z}_n) $ is a complete bipartite graph $ K_{a,b} $ with $ a=p^{k-1} $ and $ b=p^{k-1}(p-1) $.
\end{theorem}
\begin{proof}
Let $n=p^k$ for a prime number $p$ and a positive integer $k$, then $\mathbb{Z}_n$ is a local ring with its maximal ideal $\langle p \rangle=p\mathbb{Z}_n$. By Theorem \ref{theorem local implies complete bipartite}, $G_{UZ}(\mathbb{Z}_n)$ is a complete bipartite graph $K_{a,b}$ with $a=|\langle p \rangle|$ and $b=|U(\mathbb{Z}_n)|$. Since $\langle p \rangle=p\mathbb{Z}_n$, then $a=|p\mathbb{Z}_n|=\frac{n}{p}=p^{k-1}$. Using \textit{Euler's phi function}, we obtain $b=|U(\mathbb{Z}_n)|=p^{k-1}(p-1)$.
\end{proof}

Moving forward, we consider more general factorizations of $n$ and observe how the number of distinct prime divisors of $n$ influences the partite structure of the graph. This leads to the result that $G_{UZ}(\mathbb{Z}_n)$ is an $(m+1)$-partite graph when $n$ is an odd integer with $m$ distinct prime divisors.

\begin{theorem}\label{theorem Zn (m+1)-partition jhj n=p1^k1...pm^km}
If $n=p_1^{k_1}p_2^{k_2}...p_m^{k_m}$, where $p_i$ are distinct prime numbers and $k_i$ are positive integers, then $G_{UZ}(\mathbb{Z}_n)$ is a $(m+1)$-partite graph.
\end{theorem}
\begin{proof}
If $ n=p_1^{k_1}p_2^{k_2}...p_m^{k_m} $, where $ p_i $ are distinct prime numbers and $ k_i $ are positive integers, then $ \mathbb{Z}_n $ has $m$ maximal ideals, namely $ \langle p_1 \rangle, \langle p_2 \rangle, ..., \langle p_m \rangle $. By Theorem \ref{the condition for Guz is n+1 partite}, $G_{UZ}(\mathbb{Z}_n)$ is a $(m+1)$-partite graph.
\end{proof}

After establishing the general structure of $G_{UZ}(\mathbb{Z}n)$, we begin by investigating the specific conditions under which the graph contains a cycle of length 3. The following lemma shows that if $n$ is a composite odd number and not a power of any prime, then $G_{UZ}(\mathbb{Z}_n)$ necessarily contains a cycle of length 3.

\begin{lemma}\label{lemma Zn with n>6 and odd contains a cycle of length 3}
    If $n$ is a composite odd number and $n\neq p^k$ for all prime numbers $p$, then $G_{UZ}(\mathbb{Z}_n)$ contains a cycle of length 3. 
\end{lemma}
\begin{proof}
Let $n$ be a composite odd number. Hence, $n$ can be expressed as $n = p_1^{k_1} p_2^{k_2} \cdots p_r^{k_r}$
where each $p_i$ is an odd prime number. By the Chinese Remainder Theorem, we have the following isomorphism of rings
\[
\mathbb{Z}_n \cong \mathbb{Z}_{p_1^{k_1}} \times \mathbb{Z}_{p_2^{k_2}} \times \cdots \times \mathbb{Z}_{p_r^{k_r}}.
\]
Thus, every element $a \in \mathbb{Z}_n$ can be represented as $(a_1,a_2,\dots,a_r)$, where $a_i \in \mathbb{Z}_{p_i^{k_i}}$.  Here, $\bar{x}$ denotes the equivalence class of $x$ in the ring $\mathbb{Z}_n$ or $\mathbb{Z}_{p_i^{k_i}}$, while $\bar{\bar{x}}$ denotes the natural image of $\bar{x}\in\mathbb{Z}_{p_i^{k_i}}$ under the canonical projection modulo $p_i$. Under this isomorphism, each ideal $\langle \bar{p_i} \rangle \subseteq \mathbb{Z}_n$ corresponds to a set of tuples in the product ring whose $i$-th component lies in the ideal $\langle \bar{p_i} \rangle \subseteq \mathbb{Z}_{p_i^{k_i}}$, while the other components are unrestricted. In particular, the maximal ideal $\langle \bar{p_i} \rangle$ in $\mathbb{Z}_n$ is isomorphic to the ideal
\[
\mathfrak{m}_i = \mathbb{Z}_{p_1^{k_1}} \times \cdots \times (\bar{p_i} \mathbb{Z}_{p_i^{k_i}}) \times \cdots \times \mathbb{Z}_{p_r^{k_r}},
\]
which consists of all elements whose $i$-th component is zero modulo $p_i$, and other components can be arbitrary. Since $n\neq p^k$ for all prime numbers $p$, the number of prime factors of $n$ is more than one, or $r \geq 2$. Hence, $\mathbb{Z}_n$ has at least two maximal ideals. Therefore, we can choose three elements $u, x, y \in \mathbb{Z}_n$ as follows

\begin{enumerate}[label=(\roman*)]
\item $u = (\bar{1}, \bar{1},\bar{1},  \dots, \bar{1})$, which corresponds to the identity element in $\mathbb{Z}_n$. Thus, $u=\bar{1}\in U(\mathbb{Z}_n)$.
\item $x = (\bar{\bar{0}}, \bar{\bar{1}},\bar{\bar{1}},  \dots, \bar{\bar{1}})$, whose first component is congruent to $\bar{0} \pmod{p_1}$. Thus, $x \in \langle p_1 \rangle$.
\item $y = (\bar{\bar{1}}, \bar{\bar{0}},\bar{\bar{1}},  \dots, \bar{\bar{1}})$, whose second component is congruent to $\bar{0} \pmod{p_2}$. Thus, $y \in \langle p_2 \rangle$.
\end{enumerate}

\noindent As a consequence, we obtain
\begin{enumerate}[label=(\roman*)]
\item $u+x = (\bar{1},\bar{1},\bar{1},\dots,\bar{1}) + (\bar{\bar{0}},\bar{\bar{1}},\bar{\bar{1}},\dots,\bar{\bar{1}}) = (\bar{\bar{1}},\bar{\bar{2}},\bar{\bar{2}},\dots,\bar{\bar{2}})$. 
Since $p_i$ is an odd prime for all $1\leq i\leq r$, each component of the sum is a unit in $\mathbb{Z}_{p_i^{k_i}}$, so that $u+x \in U(\mathbb{Z}_n)$.
The product $ux=1x=x\in \langle p_1 \rangle$ is a zero divisor because every element in a proper ideal of a finite ring is a zero divisor. This implies that $u$ and $x$ are adjacent in $G_{UZ}(\mathbb{Z}_n)$.

\item $u+y = (\bar{1},\bar{1},\bar{1},\dots,\bar{1}) + (\bar{\bar{1}},\bar{\bar{0}},\bar{\bar{1}},\dots,\bar{\bar{1}}) = (\bar{\bar{2}},\bar{\bar{1}},\bar{\bar{2}},\dots,\bar{\bar{2}})$. 
Since $p_i$ is an odd prime for all $1\leq i\leq r$, each component of this sum is a unit in $\mathbb{Z}_{p_i^{k_i}}$ so that $u+y \in U(\mathbb{Z}_n)$.
The product $uy=1y=y\in \langle p_2 \rangle$ is a zero divisor because every element of a proper ideal in a finite ring is a zero divisor. This implies that $u$ and $y$ are adjacent in $G_{UZ}(\mathbb{Z}_n)$.
\item $x+y = (\bar{\bar{0}},\bar{\bar{1}},\bar{\bar{1}},\dots,\bar{\bar{1}}) + (\bar{\bar{1}},\bar{\bar{0}},\bar{\bar{1}},\dots,\bar{\bar{1}}) = (\bar{\bar{1}}, \bar{\bar{1}},\bar{\bar{2}},\dots,\bar{\bar{2}})$. 
Since $p_i$ is an odd prime for all $1\leq i\leq r$, each component of this sum is a unit in $\mathbb{Z}_{p_i^{k_i}}$ so that $x+y \in U(\mathbb{Z}_n)$.
The product $xy=(\bar{\bar{0}},\bar{\bar{0}},\bar{\bar{1}},\dots,\bar{\bar{1}}) \in \langle p_1 \rangle \cap \langle p_2 \rangle$ is a zero divisor. This implies that $x$ and $y$ are adjacent in $G_{UZ}(\mathbb{Z}_n)$.
\end{enumerate}
Therefore, $G_{UZ}(\mathbb{Z}_n)$ contains a cycle of length 3, namely $u\rightarrow x \rightarrow y \rightarrow u$.
\end{proof}

Conversely, we can also demonstrate that the presence of a cycle of length 3 in $G_{UZ}(\mathbb{Z}_n)$ implies that $n$ must satisfy the same condition: being a composite odd number and not a prime power. This is established by the following lemma.

\begin{lemma}\label{lemma if Guz(n) contains a cycle of length 3 then n>6 and odd}
If $G_{UZ}(\mathbb{Z}_n)$ contains a cycle of length 3, then $n$ is a composite odd number and $n\neq p^k$ for all prime numbers $p$.
\end{lemma}
\begin{proof}
We will prove it by contraposition, namely, if $n$ is an even number or $n$ is prime or $n = p^k$ for a prime number $p$, then $G_{UZ}(\mathbb{Z}_n)$ does not contain a cycle of length 3. We consider three cases:
\begin{enumerate}[label=(\roman*)]
\item If $n$ is even, then by Theorem \ref{theorem Zn bipartite iff n even}, $G_{UZ}(\mathbb{Z}_n)$ is a bipartite graph. Hence, $G_{UZ}(\mathbb{Z}_n)$ cannot contain a cycle of length 3.
\item If $n$ is a prime, then by Remark \ref{remark Zn star iff n prime}, $G_{UZ}(\mathbb{Z}_n)$ is a star graph. Hence, $G_{UZ}(\mathbb{Z}_n)$ does not contain a cycle.
\item If $n=p^k$, then by Theorem \ref{theorem Zn complete bipartite if n=p^k}, $G_{UZ}(\mathbb{Z}_n)$ is a complete bipartite graph. Hence, $G_{UZ}(\mathbb{Z}_n)$ cannot contain a cycle of length 3.
\end{enumerate}
Therefore, if $G_{UZ}(\mathbb{Z}_n)$ contains a cycle of length 3, then $n$ is an odd number, composite, and $n\neq p^k$ for all prime number $p$.
\end{proof}

Together, these results yield a complete characterization of when $G_{UZ}(\mathbb{Z}_n)$ contains a cycle of length 3, as summarized in the following theorem.

\begin{theorem}
$G_{UZ}(\mathbb{Z}_n)$ contains a cycle of length 3 if and only if $n$ is a composite odd number and $n\neq p^k$ for all prime numbers $p$.
\end{theorem}
\begin{proof}
It is clear by Lemma \ref{lemma Zn with n>6 and odd contains a cycle of length 3} and Lemma \ref{lemma if Guz(n) contains a cycle of length 3 then n>6 and odd}.
\end{proof}

With the characterization of a cycle of length 3 in hand, we next turn our attention to the existence of a cycle of length 4 in $G_{UZ}(\mathbb{Z}_n)$. There are two primary cases in which a cycle of length 4 arises: when $n$ is a prime power and when $n$ is an even number greater than 6. The following is the first lemma, which states that a prime power $n$ serves as a sufficient condition for the graph $G_{UZ}(R)$ to contain a cycle of length 4.

\begin{lemma}\label{lemma Zn with n=p^k contains a cycle of length 4}
If $n=p^k$ where $p$ is a prime number, then $G_{UZ}(\mathbb{Z}_n)$ contains a cycle of length 4.
\end{lemma}
\begin{proof}
Consider the ring $\mathbb{Z}_n$ where $n=p^k$ for some prime $p$ and some positive integer $k$. By Theorem \ref{theorem Zn complete bipartite if n=p^k}, $G_{UZ}(\mathbb{Z}_n)$ is a complete bipartite graph $K_{a,b}$ where $a=p^{k-1}$ and $b=p^{k-1}(p-1)$. Since $p \geq 2$, there are at least two vertices in each part of the partition. Thus, $G_{UZ}(\mathbb{Z}_n)$ contains a cycle of length 4.
\end{proof}

Before establishing the second sufficient condition for the unit zero-divisor graph to contain a cycle of length 4, we first apply the KST theorem to a specific case where the graph in question is a regular bipartite graph.
Suppose that $G = (X \cup Y, E(G))$ is a $t$-regular bipartite graph with $|X| = |Y| = r$. Then the total number of edges in $G$ is $|E(G)| = r \cdot t$.
By the converse of the KST theorem, in order to ensure that $G$ contains a subgraph isomorphic to $K_{2,2}$ (and hence contains a cycle of length 4), it is sufficient that 
$rt > r^{3/2} + r$ or $t > \sqrt{r} + 1$. This observation leads to the following remark, which will later serve as a key component in proving further properties of unit zero-divisor graphs.

\begin{remark}\label{remark:C4-condition}
Let $G = (X \cup Y, E(G))$ be a $t$-regular bipartite graph with $|X| = |Y| = r$. If $t > \sqrt{r} + 1$,
then $G$ contains a subgraph isomorphic to a cycle of length 4.
\end{remark}

The following lemma provides the second sufficient condition for the graph $G_{UZ}(R)$ to contain a cycle of length 4.
 
\begin{lemma}\label{lemma Zn with n>6 even contains a cycle of length 4}
If $n\geq 8$ and an even number, then $G_{UZ}(\mathbb{Z}_n)$ contains a cycle of length 4.
\end{lemma}
\begin{proof}
If $n$ is even, then $2\notin U(\mathbb{Z}_n)$. Let $|U(\mathbb{Z}_n)|=t$. By Theorem \ref{theorem if 2 is non-unit then U(R)-regular} and Theorem \ref{bipartite condition}, $G_{UZ}(\mathbb{Z}_n)$ is a $t$-reguler bipartite graph with partite sets of size $r=n/2$. By Remark \ref{remark:C4-condition}, $G_{UZ}(\mathbb{Z}_n)$ will contain a cycle of length 4 when 
\[
t > \sqrt{\frac{n}{2}} + 1.
\]
Using Euler phi function $\phi$, we have $|U(\mathbb{Z}_n)|=t=\phi(n)$. 
Recall that for any positive integer $n$, the Euler phi function $\phi(n)$ counts the number of positive integers less than or equal to $n$ that are coprime to $n$. If $n$ has the prime factorization
\[
n = p_1^{k_1} p_2^{k_2} \cdots p_s^{k_s},
\]
then the Euler phi function is given explicitly by
\[
\phi(n) = n \prod_{i=1}^s \left(1 - \frac{1}{p_i} \right).
\]
If $n=8=2^3$, then $\phi(n)=4>\sqrt{\frac{4}{2}}+1=\sqrt{\frac{n}{2}}+1$. Hence, $G_{UZ}(\mathbb{Z}_n)$ contains a cycle of length 4. Since $\phi(n)$ generally grows faster than $\sqrt{n/2} + 1$ for larger $n$, we conclude that inequality
\[
\phi(n) > \sqrt{\frac{n}{2}} + 1
\]
is satisfied for all even integers $n \geq 8$.
Therefore, if $n\geq 8$ and an even number, then $G_{UZ}(\mathbb{Z}_n)$ contains a cycle of length 4.
\end{proof}

By analyzing the structure of these cycles, we determine exactly when $G_{UZ}(\mathbb{Z}_n)$ forms a cycle graph. The complete characterization is provided in the theorem below.

\begin{theorem}\label{theorem Zn is a cycle iff n=4 and n=6}
$G_{UZ}(\mathbb{Z}_n)$ is a cycle graph if and only if $n=4$ or $n=6$.
\end{theorem}
\begin{proof}
($\Leftarrow$) In the ring $\mathbb{Z}_4$, we have $U(\mathbb{Z}_4)=\{\bar{1},\bar{3}\}$. Hence $|U(\mathbb{Z}_4)|=2$. By Theorem \ref{theorem if 2 is non-unit then U(R)-regular}, $G_{UZ}(\mathbb{Z}_4)$ is a 2-regular graph. Recall that a unit-zero divisor graph is always a simple graph, which means it contains neither loops nor multiple edges. Since there are only four vertices in $G_{UZ}(\mathbb{Z}_4)$ and all its vertices have degree 2, $G_{UZ}(\mathbb{Z}_4)\cong C_4$.

In the ring $\mathbb{Z}_6$, we have $U(\mathbb{Z}_6)=\{\bar{1},\bar{5}\}$. Hence $|U(\mathbb{Z}_6)|=2$. Since $2\notin U(\mathbb{Z}_6)$, by Theorem \ref{bipartite condition}, $G_{UZ}(\mathbb{Z}_6)$ is a bipartite graph, and by Theorem \ref{theorem if 2 is non-unit then U(R)-regular}, $G_{UZ}(\mathbb{Z}_6)$ is a 2-regular graph. With six vertices, to be a bipartite and 2-regular graph, it must be $G_{UZ}(\mathbb{Z}_4)\cong C_6$.

\noindent($\Rightarrow$) Let $n\neq 4$ and $n \neq 6$. We prove that $G_{UZ}(\mathbb{Z}_n)$ is not a cycle graph for all $n$ in addition to both.

\begin{enumerate}[label=(\roman*)]
\item If $n=1$ and $n=2$, then $G_{UZ}(\mathbb{Z}_n)$ is not a cycle since the unit-zero divisor graph is a simple graph.

\item If $n$ is a prime number, then by Remark \ref{remark Zn star iff n prime}, $G_{UZ}(\mathbb{Z}_n)$ is a star graph. Hence $G_{UZ}(\mathbb{Z}_n)$ is not a cycle.
\item If $n>6$ and $n$ is a composite, then using the Euler phi function, we have $|U(\mathbb{Z}_n)|=\phi(n)>2$. By Lemma \ref{lemma about maximum degree in general}, the maximum degree in $G_{UZ}(\mathbb{Z}_n)$ is $|U(\mathbb{Z}_n)|$. This implies that there is a vertex in $G_{UZ}(\mathbb{Z}_n)$ with a degree greater than two. Then $G_{UZ}(\mathbb{Z}_n)$ is not a cycle graph.
\end{enumerate}

Thus, if $n \neq 4$ and $n \neq 6$, then $G_{UZ}(\mathbb{Z}_n)$ is not a cycle graph. Therefore, if $G_{UZ}(\mathbb{Z}_n)$ is a cycle graph, then $n=4$ and $n=6$.
This completes the proof.
\end{proof}

Finally, we identify exactly when $G_{UZ}(\mathbb{Z}_n)$ takes the form of a path graph. These cases, while rare, highlight the simplicity and extremal structure that can arise from particular small values of $n$.
\begin{theorem}\label{theorem Zn trajectory jhj n=2 and n=3}
$G_{UZ}(\mathbb{Z}_n)$ is a path graph if and only if $n=2$ or $n=3$.
\end{theorem}
\begin{proof}
($\Leftarrow$) By Remark \ref{remark Zn star iff n prime}, $G_{UZ}(\mathbb{Z}_2)$ and $G_{UZ}(\mathbb{Z}_3)$ are star graphs. However, star graphs of orders two and three are paths. Thus, if $n=2$ or $n=3$, then $G_{UZ}(\mathbb{Z}_n)$ is a path graph.

\noindent($\Rightarrow$) Let $n\neq 2$ and $n \neq 3$. We prove that $G_{UZ}(\mathbb{Z}_n)$ is not a path graph for all $n$ in addition to both.

\begin{enumerate}[label=(\roman*)]
    \item If $n=1$, then $G_{UZ}(\mathbb{Z}_n)\cong K_1$ is not a path. 
    \item If $n>3$ and $n$ is a prime number, then by Remark \ref{remark Zn star iff n prime}, $G_{UZ}(\mathbb{Z}_n)$ is a star graph with a central degree greater than two. Hence $G_{UZ}(\mathbb{Z}_n)$ is not a path graph.
    \item If $n=4$ and $n=6$, then by Theorem \ref{theorem Zn is a cycle iff n=4 and n=6}, $G_{UZ}(\mathbb{Z}_n)$ is a cycle graph.  Hence $G_{UZ}(\mathbb{Z}_n)$ is not a path graph.
    \item If $n>6$ and $n$ is a composite, then by Lemma \ref{lemma Zn with n>6 and odd contains a cycle of length 3} and Lemma \ref{lemma Zn with n>6 even contains a cycle of length 4}, $G_{UZ}(\mathbb{Z}_n)$ contains a cycle. Hence $G_{UZ}(\mathbb{Z}_n)$ is not a path graph.
\end{enumerate}
Thus, if $n\neq 2$ and $n \neq 3$, then $G_{UZ}(\mathbb{Z}_n)$ is not a path graph.
Therefore, if $G_{UZ}(\mathbb{Z}_n)$ is a path graph, then $n=2$ or $n=3$. 
This completes the proof.
\end{proof}

\section*{Conclusion}

This study has provided a comprehensive examination of the structural properties of the unit-zero divisor graph $G_{UZ}(R)$ associated with commutative rings with identity. By analyzing various types of rings, including local rings and rings of the form $\mathbb{Z}_n$, we established results concerning planarity, bipartiteness, regularity, and Hamiltonicity. The results show a strong connection between the algebraic properties of the ring and the structure of its associated graph. In particular, properties such as the set of units, maximal ideals, and the Jacobson radical significantly influence the shape and behavior of the graph. Notably, several characterizations were provided for when $G_{UZ}(R)$ becomes a star, a cycle, a path, or a complete bipartite graph.

Despite these advances, several questions remain open. In particular, the behavior of graph-theoretic parameters such as diameter and girth of $G_{UZ}(R)$ has not yet been fully explored. Determining general bounds or exact values of these parameters for broad classes of rings could shed further light on the connectivity of the graph and the cyclic structure. Furthermore, investigating the chromatic number, clique number, and topological indices of $G_{UZ}(R)$ may reveal deeper algebraic and combinatorial patterns worthy of future research.


\subsection*{Acknowledgments}

The author gratefully acknowledges the Ministry of Higher Education, Science, and Technology of the Republic of Indonesia for the financial support provided through the Beasiswa Pendidikan Indonesia (BPI) scholarship, which has supported the completion of this study.

\bibliographystyle{ieeetr}
\bibliography{main}

\end{document}